\documentclass{amsart}
\usepackage{amscd, amsmath, amssymb, amsthm, fixmath, graphicx}
\usepackage{tikz-cd}      
\usetikzlibrary{calc}
\usepackage[all]{xy}

\newcommand{\C}{\mathbb C}
\newcommand{\catqot}{/\hskip-3pt/}
\newcommand{\E}{\mathcal{E}}

\newcommand{\F}{\mathcal{F}}
\newcommand{\G}{\mathcal{G}}
\newcommand{\GL}{\mathop{\rm GL}}
\newcommand{\id}{\mathop{\rm id}}
\renewcommand{\L}{\mathcal{L}}
\newcommand{\la}{\lambda}
\newcommand{\lma}{\longmapsto}
\newcommand{\lra}{\longrightarrow}

\newcommand{\Oh}{\mathcal{O}}
\newcommand{\ol}{\overline}

\newcommand{\q}{\quad}
\newcommand{\Q}{\mathbb Q}

\newcommand{\R}{\mathbb R}
\newcommand{\SL}{\mathop{\rm SL}}
\newcommand{\ul}{\underline}
\newcommand{\Z}{\mathbb Z}
\renewcommand{\theta}{\vartheta}

\theoremstyle{remark}
\newtheorem{Ex}{Example}[section]
\newtheorem{Rem}[Ex]{Remark}

\theoremstyle{definition}

\theoremstyle{plain}
\newtheorem{Prop}[Ex]{Proposition}
\newtheorem{Lem}[Ex]{Lemma}
\newtheorem{Cor}[Ex]{Corollary}
\newtheorem{Thm}[Ex]{Theorem}

\begin{document}
\title{Singular principal bundles on reducible nodal curves}
\author{\'Angel Luis Mu\~noz Casta\~neda}
\address{Universidad de Le\'on\\ Departamento de Mathem\'aticas\\ Campus de Vegazana, s/n\\
Le\'on 24071\\ Spain}
\email{amunc@unileon.es}
\author{Alexander H.\ W.\ Schmitt}
\address{Freie Universit\"at Berlin\\Institut f\"ur Mathematik\\ Arnimallee 3\\ D-14195 Berlin\\ Germany}
\email{alexander.schmitt@fu-berlin.de}
\begin{abstract}
Studying degenerations of moduli spaces of semistable principal bundles on smooth curves leads to the problem of studying moduli spaces on singular curves. In this note, we see that the moduli spaces of $\delta$-semistable pseudo bundles on a nodal curve become, for large values of $\delta$, the moduli spaces for semistable singular principal bundles. The latter are reasonable candidates for degenerations and a potential basis of further developments as on irreducible nodal curves. In particular, we find a notion of semistability for principal bundles on reducible nodal curves. The understanding of the asymptotic behavior of $\delta$-semistability rests on a lemma from geometric invariant theory. The results allow the construction of a universal moduli space of semistable singular principal bundles over the space of stable curves.
\end{abstract}
\maketitle
\section*{Introduction.}


The moduli space ${\mathcal U}_C(r,d)$ of semistable vector bundles of rank $r$ and degree $d$ on the smooth projective curve $C$ over the complex numbers is a basic object of modern algebraic geometry. The curve $C$ varies itself in the moduli space ${\mathcal M}_g$ of smooth curves of genus $g$, $g:=h^1(C,\Oh_C)$. There exists a moduli space ${\mathcal U}_g(r,d)\lra {\mathcal M}_g$, such that the fiber over the isomorphism class $[C']$ of a smooth curve $C'$ is isomorphic to ${\mathcal U}_{C'}(r,d)/{\rm Aut}(C')$. The moduli space ${\mathcal M}_g$ can be compactified to the moduli space $\ol{\mathcal M}_g$ of stable curves of genus $g$ \cite{DM,Gies}. Pandharipande \cite{Pan} extended ${\mathcal U}_g(r,d)\lra {\mathcal M}_g$ to a projective moduli space $\ol{\mathcal U}_g(r,d)\lra \ol{\mathcal M}_g$. The construction of this space is based on Seshadri's theory of semistable torsion free sheaves on polarized curves with ordinary double points, so-called nodal curves (\cite{Sesh}, huiti\`eme partie). It is worth noting that the compactification $\overline{\mathcal U}_g(r,d)$ is related to the study of moduli spaces of semistable vector bundles on smooth curves via degeneration techniques as in \cite{Gies2} and \cite{TiB}.  
\par 
The theory of vector bundles of rank $r$ is equivalent to the theory of principal $\GL_r$-bundles. Now, one may start with another reductive structure group $G$ and principal $G$-bundles on a smooth projective curve. The notion of semistability for these objects and the construction of moduli spaces is due to Ramanathan \cite{Ram}. Again, one may study degenerations of these moduli spaces when $C$ moves along with $\ol{\mathcal M}_g$.
We refer to \cite{Fal}, \cite{SchSurv}, \cite{Sesh2}, and \cite{Sol} for some work on the problem of degenerating moduli spaces of principal $G$-bundles on smooth curves. Recently, important progress was made by Balaji in \cite{Bal}. He found an intrinsic notion of semistability for so-called laced torsors (\cite{Bal}, Definition 12.5) (Interestingly, this notion can be phrased in the terms that we are using in this note (\cite{Bal}, Theorem 12.13), so that the situation becomes somewhat similar to the one on smooth curves). Furthermore, Balaji's moduli spaces provide flat degenerations of moduli spaces of semistable principal bundles on smooth curves (\cite{Bal}, Theorem 1.1). 

In this note, we will pursue the approach discussed in \cite{SchSurv}. The first author has been successful in generalizing many results from irreducible nodal curves to arbitrary nodal curves \cite{MunThesis,MunPap}. The objects we are dealing with are called pseudo $G$-bundles\footnote{We use the terminology of \cite{SchNod2} rather than of \cite{MunPap} in order to shorten the technical terms.} (see Section \ref{sect:SingPrinz}). These are special instances of decorated torsion free sheaves and, as such, they have a quite natural notion of semistability which depends on a positive rational number $\delta$. The key is to understand the objects which are semistable for large values of the stability parameter $\delta$. The second author proved that the moduli space of $\delta$-semistable pseudo $G$-bundles is isomorphic to the moduli space of principal $G$-bundles constructed by Ramanathan provided $\delta$ is large enough \cite{SchMani}. Furthermore, in case the base curve is an irreducible nodal curve and if $\delta$ is large enough, every $\delta$-semistable pseudo $G$-bundle determines a principal $G$-bundle over a generic open subset \cite{Bhosle2,SchNod2}.
We will present a lemma from geometric invariant theory (Lemma \ref{lem:GIT}) which is essential for understanding the asymptotic behavior of $\delta$-semistability on reducible curves. Then, we will adapt the techniques from \cite{GLSS} to prove a crucial boundedness result and give a description of asymptotic semistability (Theorem \ref{thm:MainAux} and Theorem \ref{thm:MainAux2}). 
In the final section, we will apply the results to pseudo bundles on reducible nodal curves. We will see that the moduli spaces constructed in \cite{MunPap} parameterize, for large values of the stability parameter, objects which are principal bundles on a dense open subset of the curve (which are called singular principal bundles) and satisfy a nice semistability condition (Theorem \ref{thm:moduli-stb}). In particular, we find a notion of semistability for (genuine) principal bundles on reducible nodal curves which depends on a faithful representation of the structure group. To our knowledge, such a notion has not appeared in the literature, yet. We will also prove that, for a faithful representation of the structure group $G$ with values in a symplectic or an orthogonal group, a singular principal $G$-bundle will be a principal $G$-bundle on the smooth locus of the curve $X$ (Theorem \ref{thm:good-choice}), generalizing the findings of \cite{SchNod2}.
\par
For a fixed value $\delta\in\Q_{>0}$ and a fixed faithful representation $\sigma\colon G\lra \SL(V)$ of the structure group $G$, the first author has constructed in \cite{MunUniv} a relative moduli space $\Theta\colon {\mathcal S}{\mathcal P}{\mathcal B}(\sigma)^{\delta\hbox{-}\rm ss}_{g}\lra \ol{\mathcal M}_g$, such that the fiber of $\Theta$ over the isomorphy class of a stable curve $C$ is isomorphic to ${\mathcal S}{\mathcal P}{\mathcal B}(\sigma)^{\delta\hbox{-}\rm ss}_{C,\chi}/{\rm Aut}(C)$, ${\mathcal S}{\mathcal P}{\mathcal B}(\sigma)^{\delta\hbox{-}\rm ss}_{C,\chi}$ the moduli space constructed in \cite{MunUniv}, for $\chi=\dim_\C(V)\cdot (1-g)$. The improved estimates of the present note are used to show that there is a rational number $\delta_\infty\in \Q_{>0}$, such that, for every $\delta>\delta_\infty$, 
${\mathcal S}{\mathcal P}{\mathcal B}(\sigma)^{\delta\hbox{-}\rm ss}_{g}$ parameterizes pairs given by a stable curve of genus $g$ and a semistable singular principal $G$-bundle. This means that, for $\delta>\delta_\infty$ and a smooth projective curve $C$, $\Theta^{-1}([C])$ is the moduli space of semistable principal $G$-bundles on $C$, divided by the action of the finite group ${\rm Aut}(C)$, and, for $\delta>\delta_\infty$ and a singular projective curve $C$, $\Theta^{-1}([C])$ is the moduli space for semistable singular principal $G$-bundles, divided by ${\rm Aut}(C)$. If, in addition, we choose $\sigma$ to have values in a symplectic or an orthogonal group, the latter moduli space parameterizes objects which are principal $G$-bundles away from the nodes. So, ${\mathcal S}{\mathcal P}{\mathcal B}(\sigma)^{\delta\hbox{-}\rm ss}_{g}$ is, for $\delta>\delta_\infty$, a reasonable generalization of Pandharipande's moduli space \cite{Pan}.
\par 
To conclude, we note that compactifications of the universal moduli space of semistable principal $\SL_r(\C)$-bundles over $\ol{\mathcal M}_g$ appeared in the context of conformal blocks. In fact, Manon \cite{Man} and Belkale/Gibney \cite{BG} constructed such compactifications, using algebras of conformal blocks. They did not give a modular interpretation (see, however, \cite{BG}, Section 11.2), and it would be interesting to relate our moduli spaces to these compactifications. The arguments by Belkale and Gibney (e.g., Section 3 of \cite{BG}) and the work of Balaji (\cite{Bal}, in particular Section 12) show that the moduli spaces we propose to construct may be useful even if they have the drawback of depending on the choice of a faithful representation of the structure group. It would be interesting to see to which extent the results of \cite{Bal} generalize to reducible curves, e.g., whether there is a Hilbert compactification as in \cite{SchHilb}.

\section{A lemma from geometric invariant theory}

Let ${\mathbb K}$ be a field. Suppose we are given tuples $\ul r=(r_1,...,r_t)$ and $\ul \ell=(\ell_1,...,\ell_t)$ of positive integers. We associate with these data the embedding
\begin{equation}
\label{eq:BlockEmb}
\iota_{\ul\ell}\colon {\GL}_{r_1}({\mathbb K})\times\cdots\times {\GL}_{r_t}({\mathbb K})\hookrightarrow {\GL}_R({\mathbb K}),\q R:=\ell_1\cdot r_1+\cdots+\ell_t\cdot r_t,
\end{equation}
which maps a tuple $\ul g=(g_1,...,g_t)$ to the block diagonal matrix in which the first $\ell_1$ blocks are copies of $g_1$, the blocks $\ell_1+1$ up to $\ell_1+\ell_2$ are copies of $g_2$, and so on. Define
\[
H :=\iota_{\ul\ell}^{-1}({\SL}_R({\mathbb K})).
\]
\par 
Assume further that $\varrho_i\colon \GL_{r_i}({\mathbb K})\lra \GL(W_i)$ is a homogeneous representation and let $h_i$ be its degree, $i = 1,...,t$. Note that the representation $\varrho_i$ may be viewed as a representation of ${\GL}_{r_1}({\mathbb K})\times\cdots\times {\GL}_{r_t}({\mathbb K})$ on $W_i$, $i=1,...,t$. The direct sum of these representations is a representation of ${\GL}_{r_1}({\mathbb K})\times\cdots\times {\GL}_{r_t}({\mathbb K})$ on
\[
W:= W_1\oplus\cdots\oplus W_t.
\]
By means of restriction, we get the representation $\varrho\colon H\lra \GL(W)$. We would like to investigate the notion of $\varrho$-semistability on $W$.
\begin{Lem}
\label{lem:GIT}
Assume that the field ${\mathbb K}$ is perfect and that $h_i$ is positive, $i=1,...,t$. Let $w = (w_1 , ... , w_t)\in W\setminus \{0\}$. Then, the following assertions are equivalent.
\begin{itemize}
\item[\rm i)] The point $w$ is $\varrho$-semistable.
\item[\rm ii)] We have $w_i\neq 0$, and $w_i$ is $\widetilde{\varrho}_i$-semistable, $\widetilde{\varrho}_i:=\varrho_{i|{\SL}_{r_i}({\mathbb K})}$, $i=1,...,t$.
\end{itemize}
\end{Lem}
\begin{Rem}
\label{rem:AssField}
i) The statement of the lemma holds also true, if the degree $h_i$ is negative, for $i=1,...,t$.
\par 
ii) Observe that the notion of $\varrho$-semistability does not depend on the tuple $\ul\ell=(\ell_1,...,\ell_t)$ of positive integers which enters the definition of $H$.
\end{Rem}
Let $r$ be a positive number. By ${\rm det}\colon \GL_r({\mathbb K})\lra {\mathbb G}_m({\mathbb K})$, we denote the determinant homomorphism. We introduce the one parameter subgroup
\begin{eqnarray*}
\zeta\colon {\mathbb G}_m({\mathbb K}) &\lra& {\GL}_r({\mathbb K})
\\
z &\lma& z\cdot E_r.
\end{eqnarray*}
For a one parameter subgroup $\la\colon {\mathbb G}_m({\mathbb K})\lra \GL_r({\mathbb K})$, we get the homomorphism $\det\circ \la\colon {\mathbb G}_m({\mathbb K})\lra {\mathbb G}_m({\mathbb K})$. It is of the form $z\lma z^\gamma$, $z\in {\mathbb G}_m({\mathbb K})$, for some integer $\gamma\in\Z$. In the Hilbert--Mumford criterion, we may always replace $\la$ by a positive multiple, so that we may assume that $\gamma$ is divisible by $r$. Then, written in additive notation,
\[
\widetilde{\la}:=\la-\frac{\gamma}{r}\cdot \zeta
\]
is a one parameter subgroup of $\SL_r({\mathbb K})$. So, we may view a one parameter subgroup $\la$ of $\GL_r({\mathbb K})$ as a pair $(\widetilde{\la},\delta)$, consisting of a one parameter subgroup $\widetilde{\la}$ of $\SL_r({\mathbb K})$ and an integer $\delta\in\Z$. Finally, let $\varrho\colon \GL_r({\mathbb K})\lra \GL(W)$ be a homogeneous representation, $h$ its degree, and $\widetilde{\varrho}:=\varrho_{|\SL_r({\mathbb K})}$ its restriction to the special linear group. For a point $w\in W$ and a one parameter subgroup $\la=(\widetilde{\la},\delta)$ of $\GL_r({\mathbb K})$, we find the formula 
\begin{equation}
\label{eq:SLToGL}
\mu_\varrho(\la,w)=\mu_{\widetilde{\varrho}}(\widetilde{\la},w)+h\cdot \delta.
\end{equation}
\begin{proof}[Proof of Lemma {\rm \ref{lem:GIT}}] 
Note that the assumption on ${\mathbb K}$ implies that semistability may be tested with the Hilbert--Mumford criterion (see \cite{SchBook}, Section 1.7.1, for a brief discussion). 
\par 
A one parameter subgroup of $H$ is a tuple $\nu=(\nu_1,...,\nu_t)$ with $\nu_i=(\widetilde{\nu}_i,\delta_i)$ a one parameter subgroup of $\GL_{r_i}({\mathbb K})$, $i=1,...,t$, such that 
\begin{equation}
\label{eq:NullRel}
\ell_1\cdot \delta_1+\cdots+\ell_t\cdot \delta_t=0.
\end{equation}
We will also use the obvious formula
\begin{equation}
\label{eq:Obviou}
\mu_\varrho(\nu,w)=\max\bigl\{\, \mu_{\varrho_i}(\nu_i,w_i)\,|\,w_i\neq 0\,\bigr\}.
\end{equation}
\par 
{\bfseries ii) $\boldsymbol{\Longrightarrow}$ i).}  Let $\nu=(\nu_1,...,\nu_t)$, $\nu_i=(\widetilde{\nu}_i,\delta_i)$, $i=1,...,t$, be a one parameter subgroup of $H$. Because of (\ref{eq:NullRel}), there is an index $i_0$ with $\delta_{i_0}\ge 0$. Since we assume that $w_{i_0}$ is $\widetilde{\varrho}_{i_0}$-semistable, we have 
\[
\mu_{\widetilde{\varrho}_{i_0}}(\widetilde{\nu}_{i_0}, w_{i_0})\ge 0.
\]
So, using (\ref{eq:Obviou}) and (\ref{eq:SLToGL}), we see
\[
\mu_\varrho(\nu,w)\ge \mu_{\widetilde{\varrho}_{i_0}}(\widetilde{\nu}_{i_0}, w_{i_0})+h_{i_0}\cdot \delta_{i_0}\ge 0.
\]
The Hilbert--Mumford criterion shows that $w$ is $\varrho$-semistable.
\par 
{\bfseries i) $\boldsymbol{\Longrightarrow}$ ii).} Assume first that $w_1=0$. Since $w\neq 0$, there is an index $i_0\in\{\, 2,...,t\,\}$ with $w_{i_0}\neq 0$. Let us consider the case $i_0=2$. Set $\delta_1:=\ell_2$, $\delta_2:=-\ell_1$, and $\delta_i=0$, $i=3,...,t$. Then, (\ref{eq:NullRel}) is satisfied. Next, define $\nu_i:=(0,\delta_i)$, $i=1,...,t$, and $\nu=(\nu_1,...,\nu_t)$. This is a one parameter subgroup of $H$ with 
\[
\mu_\varrho(\nu,w)=h_2\cdot \delta_2=-h_2\cdot \ell_1<0.
\]
Now, suppose that $w_1\neq 0$ and that $w_1$ is not $\widetilde{\varrho}_1$-semistable. Pick a one parameter subgroup $\widetilde{\nu}_1$ of $\SL_{r_1}({\mathbb K})$ with $\mu_{\widetilde{\varrho}_1}(\widetilde{\nu}_1,w_1)<0$, a positive rational number $\eta_1$, such that
\[
\mu_{\widetilde{\varrho}_1}(\widetilde{\nu}_1,w_1)+h_1\cdot \eta_1<0,
\]
and negative rational numbers $\eta_i$, $i=2,...,t$, with
\[
\ell_1\cdot \eta_1+\cdots+\ell_t\cdot \eta_t=0.
\]
Choose $s>0$, such that $\delta_i=s\cdot \eta_i$ is an integer, $i=1,...,t$, set $\nu_1:=(s\cdot \widetilde{\nu}_1, \delta_1)$, $\nu_i:=(0,\delta_i)$, $i=2,...,t$, and $\nu:=(\nu_1,...,\nu_t)$; this is a one parameter subgroup of $H$. Denoting by $j_1<\cdots<j_u$ the elements of $\{\, i\in\{\, 2,...,t\,\}\,|\, w_i\neq 0\,\}$ and using (\ref{eq:Obviou}) and (\ref{eq:SLToGL}), we deduce
\[
\mu_\varrho(\nu,w) =
\max\bigl\{\, s\cdot\mu_{\widetilde{\varrho}_1}(\widetilde{\nu}_1,w_1)+h_1\cdot\delta_1, h_{j_1}\cdot\delta_{j_1},...,h_{j_u}\cdot\delta_{j_u}\,\bigr\}<0.
\]
In both cases, we have found a contradiction to the hypothesis that $w$ be $\varrho$-semistable.
\end{proof}

\section{Parabolic swamps}\label{sec:ParSwa}

Let $C$ be a (possibly disconnected) reduced projective algebraic curve and $N:=\{\, \mu_1,\nu_1,\allowbreak ...,\mu_c,\nu_c\,\}$ a set of $2c$ distinct smooth points on $C$. The case $c=0$, i.e., $N=\varnothing$ is allowed. Finally, we fix an ample line bundle $\L$ on $C$. Define $\ul\ell=(\ell_1,...,\ell_t)$ with $\ell_i:=\deg(\L_{|C_i})$, $i=1,...,t$. For a coherent $\Oh_C$-module $\G$, the \it multirank \rm is defined as 
$\ul r(\G):=({\rm rk}(\G_{|C_1}),...,{\rm rk}(\G_{|C_t}))$, 
and the \it total rank \rm as ${\rm trk}(\G):=\sum\limits_{i=1}^t \ell_i\cdot {\rm rk}(\G_{|C_i})$. We say that $\G$ has \it uniform rank \rm $r$, if ${\rm rk}(\G_{|C_i})=r$, $i=1,...,t$.
\begin{Rem}
\label{rem:EnforceRiemannRoch}
If we define
\[
{\rm rk}_{\ul\ell}(\G):=\frac{{\rm trk}(\G)}{\sum\limits_{i=1}^t \ell_i}
\]
and 
\[
{\deg}_{\ul\ell}(\G)=\chi(X,\G)-{\rm rk}_{\ul\ell}(\G)\cdot \bigl(1-p_a(X)\bigr),
\]
then
\begin{itemize}
\item the formula ${\rm rk}_{\ul\ell}(\G)=r$ holds true, if $\G$ has uniform rank $r$ on $X$,
\item the Riemann--Roch formula $\chi(X,\G)={\deg}_{\ul\ell}(\G)+{\rm rk}_{\ul\ell}(\G)\cdot (1-p_a(X))$ is satisfied, by definition, and 
\item the degree behaves additively on short exact sequences of coherent $\Oh_X$-modules.
\end{itemize}
\end{Rem}
Let $\E$ be a torsion free coherent $\Oh_C$-module. A \it generalized parabolic structure \rm on $\E$ is the datum of a tuple $\ul q=(q_i\colon \E_{|\{\mu_i\}}\oplus \E_{|\{\nu_i\}}\lra R_i, i=1,...,c)$ in which $R_i$ is a complex vector space and $q_i$ is a surjective $\C$-linear map, $i=1,...,c$. The tuple $\ul t(\ul q):=(\dim_\C(R_i), i=1,...,c)$ is the \it type \rm of the generalized parabolic structure. A submodule $\F\subset \E$ is \it saturated\rm, if the quotient module $\E/\F$ is torsion free, as well. Suppose that $(\E,\ul q)$ is a torsion free $\Oh_C$-module which is endowed with a generalized parabolic structure and $\F$ is a saturated submodule. Then, $\ul q_\F=(q_{i,\F}, i=1,...,c)$ with 
\[
q_{i,\F}\colon \F_{|\{\mu_i\}}\oplus \F_{|\{\nu_i\}}\lra S_i,\q S_i:=q_i(\F_{|\{\mu_i\}}\oplus \F_{|\{\nu_i\}}),\q i=1,...,c,
\]
is a generalized parabolic structure on $\F$. Let us also fix a tuple $\ul\kappa=(\kappa_1,...,\kappa_c)$ of positive real numbers. For a torsion free $\Oh_C$-module $(\E,\ul q)$ with generalized parabolic structure, we set 
\[
\chi_{\ul\kappa}(\E,\ul q):=\chi(\E)-\sum_{i=1}^c \kappa_i\cdot \dim_\C(R_i).
\]
A \it weighted filtration \rm of $\E$ is a pair $(\E_\bullet, m_\bullet)$ which consists of a filtration
\[
0=:\E_0\subsetneq \E_1\subsetneq \cdots\subsetneq \E_s\subsetneq \E_{s+1}:=\E
\]
of $\E$ by saturated submodules and a tuple $m_\bullet=(m_1,...,m_s)$ of positive rational numbers. We define
\[
\chi_{\ul\kappa}(\E_\bullet, m_\bullet;\ul q):=\sum_{i=1}^s m_i\cdot \bigl(\chi_{\ul\kappa}(\E,\ul q)\cdot {\rm trk}(\E_i)-\chi_{\ul \kappa}(\E_i,\ul q_{\E_i})\cdot {\rm trk}(\E)\bigr).
\]
\par 
Next, let us also fix positive integers $a$ and $b$. Then, an \it $(a,b)$-GPS-swamp \rm or, simply, a \it GPS-swamp \rm is a tuple $(\E,\ul q,\varphi)$ which consists of a torsion free $\Oh_C$-module $\E$, a generalized parabolic structure $\ul q$ on $\E$, and a non-trivial homomorphism $\varphi\colon (\E^{\otimes a})^{\oplus b}\lra \Oh_C$. Fix the total rank $\alpha$ and set 
\begin{equation}
\Gamma^{(i)}:=\bigl(\underbrace{i-\alpha,...,i-\alpha}_{i\times}, \underbrace{i,...,i}_{(\alpha-i)\times}\bigr),\q i=1,...,\alpha-1.
\end{equation}
Suppose that $(\E,m_\bullet)$ is a weighted filtration of $\E$ and define
\[
\Gamma_\bullet:=(\Gamma_1,...,\Gamma_\alpha):=\sum_{i=1}^s m_i\cdot \Gamma^{({\rm trk}(\E_i))},\q
\Gamma(i):=\Gamma_{{\rm trk}(\E_i)},\q i=1,...,s+1,
\]
as well as 
\[
\mu(\E_\bullet, m_\bullet;\varphi):=-\min_{\{\, 1,...,s+1\}^{\times a}}\bigl\{\, \Gamma(i_1)+\cdots+\Gamma(i_a)\,|\, \varphi_{|(\E_{i_1}\otimes\cdots\otimes \E_{i_a})^{\oplus b}}\not\equiv 0\,\bigr\}.
\]
Given a positive rational number $\delta\in\Q_{>0}$, we say that the GPS-swamp $(\E,\ul q,\varphi)$ is \it $(\ul\kappa,\delta)$-(semi)stable\rm, if the inequality
\[
\chi_{\ul\kappa}(\E_\bullet,m_\bullet,\ul q)+\delta\cdot \mu(\E_\bullet, m_\bullet;\varphi) (\ge) 0
\]
is satisfied, for every weighted filtration $(\E_\bullet, m_\bullet)$ of $\E$.
\begin{Rem}
\label{rem:FinMan}
We fix the total rank $\alpha$. Introduce $R$ as the set of all tuples $\ul r=(r_1,...,r_{s+1})$ of natural numbers with $0<r_1<\cdots<r_s<r_{s+1}=\alpha$ and $s\ge 0$. For $\ul r=(r_1,...,r_{s+1})\in R$, let ${\mathcal C}(\ul r)\subset \R^\alpha$ be the cone spanned by $\Gamma^{(r_1)},...,\Gamma^{(r_s)}$. It consists of all vectors $(v_1,....,v_\alpha)$ that satisfy $v_1=\cdots=v_{r_1}\le v_{r_1+1}=\cdots=v_{r_2}\le v_{r_2+1}=\cdots \le v_{r_s+1}=\cdots=v_\alpha$ and $v_1+\cdots+v_\alpha=0$. For $s\in \{\, 0,...,\alpha-1\,\}$, let ${\mathcal P}(s)$ be the set of non-empty subsets of $\{\, 1,...,s+1\,\}^{\times a}$. Then, for $\ul r=(r_1,...,r_{s+1})\in R$ and $P\in {\mathcal P}(s)$,
\begin{eqnarray*}
\Phi(\ul r, P)\colon {\mathcal C}(\ul r) &\lra& \R
\\
(v_1,...,v_\alpha) &\lma& -\min\bigl\{\, v_{r_{i_1}}+\cdots+v_{r_{i_a}}\,|\, (i_1,...,i_a)\in P\,\bigr\}
\end{eqnarray*}
is a piecewise linear function. Note that the sets $R$ and ${\mathcal P}(1)\sqcup\cdots\sqcup {\mathcal P}(\alpha-1)$ are finite, so that the above rule creates only finitely many piecewise linear functions, and that, in the notation from above, $\mu(\E_\bullet, m_\bullet;\varphi)=\Phi(\ul r, P)(\Gamma_\bullet)$ for $\ul r=({\rm trk}(\E_1),...,{\rm trk}(\E_{s+1}))$ and $P:=\{\, (i_1,...,i_a)\in \{\, 1,...,s+1\,\}^{\times a}\,|\,\allowbreak \varphi_{|(\E_{i_1}\otimes\cdots\otimes \E_{i_a})^{\oplus b}}\not\equiv 0\,\}$.
\end{Rem} 
An $(a,b)$-GPS-swamp $(\E,\ul q,\varphi)$ is \it $\ul\kappa$-asymptotically (semi)stable\rm, if
\begin{itemize}
\item $\mu(\E_\bullet, m_\bullet;\varphi)\ge 0$ and
\item $\mu(\E_\bullet, m_\bullet;\varphi)=0$ implies $\chi_{\ul\kappa}(\E_\bullet, m_\bullet;\ul q)(\ge)0$,
\end{itemize}
for every weighted filtration $(\E_\bullet, m_\bullet)$ of $\E$.
\subsection*{The generic tensor field of a swamp. Generic semistability}
In this section, we would like to give an interpretation of the first condition in the definition of asymptotic (semi)stability in terms of geometric invariant theory.
\par 
Let $C$ be a curve as above and fix positive integers $r,a,b$. We let ${\mathbb K}_i$ be the function field of the irreducible component $C_i$, $\eta_i:={\rm Spec}({\mathbb K}_i)$ its generic point, $i=1,...,t$, and
\[
\Xi:=\{\eta_1\}\sqcup\cdots\sqcup \{\eta_t\}.
\]
Suppose $\E$ is a torsion free $\Oh_C$-module. The restriction of $\E$ to $\Xi$ is denoted by ${\mathbb E}$. We view ${\mathbb E}$ as the tuple $({\mathbb E}(1),...,{\mathbb E}(t))$ with ${\mathbb E}(i)=\E_{|\{\eta_i\}}$, $i=1,...,t$. Observe that ${\mathbb E}(i)$ is a ${\mathbb K}_i$-vector space, $i=1,...,t$. Introduce 
\[
W'_i:={\rm Hom}_{{\mathbb K}_i}\bigl(({\mathbb E}(i)^{\otimes a})^{\oplus b}, {\mathbb K}_i\bigr)
\]
and let $\varrho_i\colon \GL({\mathbb E}(i))\lra \GL(W'_i)$ be the natural representation of $\GL({\mathbb E}(i))$ on $W'_i$, $i=1,...,t$. 
\par 
Next, let $\varphi\colon (\E^{\otimes a})^{\oplus b}\lra\Oh_C$ be a tensor field on $\E$. The restriction of $\varphi$ to $\eta_i$ yields a point $w'_i\in W'_i$, $i=1,...,t$. Since the transcendence degree of ${\mathbb K}_i$ is one, we may present that field as an algebraic extension of $\C(t)$ and, therefore, embed it into the algebraic closure ${\mathbb K}$ of $\C(t)$, $i=1,...,t$. Define 
\[
W_i:=W_i'\mathop{\otimes}_{{\mathbb K}_i}{\mathbb K}
\]
and let $w_i\in W_i$ be the point defined by $w'_i\in W'_i$, $i=1,...,t$. The point 
\[
w=(w_1,...,w_t)\in W_1\oplus\cdots\oplus W_t=:W
\]
is called the \it generic tensor field \rm associated with $\varphi$. Define $\varrho\colon H\lra \GL(W)$ as the natural representation. If $w$ is $\varrho$-semistable, we say that $\varphi$ is \it generically semistable\rm.
\begin{Lem}
\label{lem:GenSemStab}
The condition that $\mu(\E_\bullet, m_\bullet;\varphi)\ge 0$ holds for every weighted filtration $(\E_\bullet, m_\bullet)$ is equivalent to the fact that $\varphi$ is generically semistable.
\end{Lem}
\begin{proof} 
A weighted filtration $(\E_\bullet, m_\bullet)$ of $\E$ yields the filtration
\[
0={\mathbb E}_0\subsetneq {\mathbb E}_1\subsetneq \cdots\subsetneq {\mathbb E}_s\subsetneq {\mathbb E}_{s+1}={\mathbb E}
\]
of ${\mathbb E}$ and the weights
\[
\Gamma(1)<\cdots<\Gamma(s+1).
\]
For each index $j\in\{\, 1,...,t\,\}$, we get the filtration 
\[
0={\mathbb E}_0(j)\subseteq {\mathbb E}_1(j)\subseteq \cdots\subseteq {\mathbb E}_s(j)\subseteq {\mathbb E}_{s+1}(j)={\mathbb E}(j)
\]
of ${\mathbb E}(j)$. We remove improper inclusions and get 
\[
{\mathbb E}^j_\bullet:\q 0={\mathbb E}^j_0\subsetneq {\mathbb E}^j_1\subsetneq\cdots\subsetneq {\mathbb E}^j_{s_j}\subsetneq {\mathbb E}^j_{s_j+1}={\mathbb E}(j).
\]
Note that, for some indices $j\in \{\, 1,...,t\,\}$, we may have 
\[
{\mathbb E}^j_\bullet:\q 0\subsetneq {\mathbb E}(j).
\]
We define 
\[
\Gamma^j_k:=\min\bigl\{\,\Gamma(i)\,|\,i\in\{\, 1,...,s+1\,\}\wedge {\mathbb E}_i(j)={\mathbb E}^j_k\,\bigr\},\q k=1,...,s_j+1,\ j=1,...,t,
\]
and 
\[
\Gamma^j_\bullet:=(\Gamma^j_1,...,\Gamma^j_{s_j+1}).
\]
By the procedure described, e.g., in \cite{SchBook}, Example 1.5.1.36 - adapted to one parameter subgroups of $\GL_r$ - the weighted flag $({\mathbb E}^j_\bullet,\Gamma^j_\bullet)$ comes from some one parameter subgroup $\la_j\colon {\mathbb G}_m({\mathbb K}_j)\lra \GL({\mathbb E}(j))$, $j=1,...,t$. Although $\la_j$ is not uniquely defined, the value $\mu(\la_j,w'_j)$ is, $j=1,...,t$ (\cite{SchBook}, Proposition 1.5.1.35). It is now straightforward to check that 
\begin{equation}
\label{eq:CompMu}
\mu(\E_\bullet, m_\bullet;\varphi)=\max\bigl\{\, \mu(\la_1,w'_1),...,\mu(\la_t,w'_t)\,\bigr\}.
\end{equation}
This shows that $\mu(\E_\bullet, m_\bullet;\varphi)\ge 0$ holds for every weighted filtration $(\E_\bullet, m_\bullet)$ of $\E$, if $\varphi$ is generically semistable.
\par 
To prove the converse, we have to construct a weighted filtration $(\E_\bullet,m_\bullet)$ from the datum of a certain tuple $(\la_1,...,\la_t)$ in which $\la_j$ is a one parameter subgroup of $\GL({\mathbb E}(j))$, $j=1,...,t$. So, first suppose that $j\in\{\, 1,...,t\,\}$ and $\la_j\colon {\mathbb G}_m({\mathbb K}_j)\lra \GL({\mathbb E}(j))$ is a one parameter subgroup. It yields a filtration
\[
{\mathbb E}^j_\bullet:\q 0={\mathbb E}^j_0\subsetneq {\mathbb E}^j_1\subsetneq\cdots\subsetneq {\mathbb E}^j_{s_j}\subsetneq {\mathbb E}^j_{s_j+1}={\mathbb E}(j).
\]
and a vector $\Gamma_\bullet^j$ consisting of weights
\[
\Gamma_1^j<\cdots<\Gamma_{s_j+1}^j.
\]
Next, suppose that we are given a tuple $(\la_1,...,\la_t)$ in which $\la_j\colon {\mathbb G}_m({\mathbb K}_j)\lra \GL({\mathbb E}(j))$ is a one parameter subgroup, $j=1,...,t$, such that the weighted filtrations $({\mathbb E}^j_\bullet, \Gamma_\bullet^j)$, $j=1,...,t$, satisfy the constraint
\begin{equation}
\label{eq:Constraint}
\sum_{j=1}^t \ell_j\cdot\Bigl(\sum_{k=1}^{s_j+1} \bigl(\dim_{{\mathbb K}_j}({\mathbb E}_k^j)-\dim_{{\mathbb K}_j}({\mathbb E}_{k-1}^j)\bigr)\cdot \Gamma^j_k\Bigr)=0.
\end{equation}
Now, we let 
\[
\Gamma(1)<\cdots <\Gamma(s+1)
\]
be the distinct weights occuring among the $\Gamma^j_k$, $k=1,...,s_j+1$, $j=1,...,t$. For $i\in \{\, 1,...,s+1\,\}$ and $j\in \{\,1,...,t\,\}$, we set 
\[
{\mathbb E}_i(j):={\mathbb E}_k^j
\]
where 
\[
k:=\max\bigl\{\,\kappa\,|\, \kappa\in\{\,1,...,s_j+1\,\}\wedge \Gamma_k^j\le \Gamma(i)\bigr\}.
\]
Set ${\mathbb E}_i:=({\mathbb E}_i(1),...,{\mathbb E}_i(t))$, $i=1,...,s+1$. So far, we have constructed a filtration 
\[
0={\mathbb E}_0\subsetneq {\mathbb E}_1\subsetneq \cdots\subsetneq {\mathbb E}_s\subsetneq {\mathbb E}_{s+1}={\mathbb E}
\]
of ${\mathbb E}$. Let $\iota\colon \Xi\lra C$ be the inclusion. Now, the natural homomorphism
\[
\E\lra\iota_\star({\mathbb E})
\]
is injective, because $\E$ is torsion free. We define 
\[
\widetilde{\E}_i:=\E\cap \iota_\star({\mathbb E}_i)
\]
and 
\[
\E_i:={\rm Ker}\bigl(\E\lra (\E/\widetilde{\E}_i)/\ul{\rm Tors}(\E/\widetilde{\E}_i)\bigr),\q i=1,...,t.
\]
So,
\[
\E_\bullet:\q 0=\E_0\subsetneq \E_1\subsetneq \cdots\subsetneq \E_s\subsetneq \E_{s+1}
\]
is a filtration of $\E$ by saturated subsheaves. We also define $m_\bullet=(m_1,...,m_s)$ by
\[
m_i:=\frac{\Gamma(i+1)-\Gamma(i)}{\alpha},\q i=1,...,s.
\]
By Constraint \eqref{eq:Constraint}, $\Gamma_\bullet=\sum\limits_{i=1}^s m_i\cdot \Gamma^{({\rm trk}(\E_i))}$ satisfies $\Gamma_{{\rm trk}(\E_i)}=\Gamma(i)$, $i=1,...,s+1$. It follows that Formula \eqref{eq:CompMu} holds true. First, we observe that the assumption implies that the restriction of $\varphi$ to $C_j$ is non-zero, $j=1,...,t$. We use the argument from the proof of Lemma \ref{lem:GIT}. If, for example, the restriction of $\varphi$ to $C_1$ were zero, we would pick a negative integer $\delta_1$ which is divisible by $\ell_1$ and set 
\[
\delta_2=-\frac{\delta_1}{\ell_1}\cdot (\ell_2+\cdots+\ell_t).
\]
Define the one parameter subgroups $\la_1\colon {\mathbb G}_m({\mathbb K}_1)\lra \GL({\mathbb E}(1))$, $z\lma z^{\delta_1}\cdot \id_{{\mathbb E}(1)}$, and $\la_j\colon {\mathbb G}_m({\mathbb K}_j)\lra\allowbreak \GL({\mathbb E}(j))$, $z\lma z^{\delta_2}\cdot \id_{{\mathbb E}(j)}$, $j=2,...,t$. For the weighted filtration $(\E_\bullet,m_\bullet)$ that was constructed above, we get 
\[
\mu(\E_\bullet,m_\bullet;\varphi)=-a\cdot \delta_2<0.
\]
This contradicts the assumption. Next, take a one parameter subgroup $\la_1\colon {\mathbb G}_m({\mathbb K}_1)\allowbreak\lra \SL_r({\mathbb E}(1))$ and use $\la=(\la_1,0,...,0)$. Since $\la_1$ maps to the special linear group, Constraint \eqref{eq:Constraint} is satisfied. We get a weighted filtration $(\E_\bullet,m_\bullet)$ with 
\[
\mu(\la_1,w'_1)=\mu(\E_\bullet,m_\bullet;\varphi)\ge 0.
\]
This shows that $w_1'$ is semistable with respect to the action of the special linear group $\SL({\mathbb E}(1))$. Since semistability is preserved under field extensions, we infer that $w_1$ is semistable with respect to the action of the special linear group $\SL({\mathbb E}(1)\mathop{\otimes}\limits_{{\mathbb K}_1} {\mathbb K})$. In a similar fashion, we show that $w_j$ is semistable with respect to the action of the special linear group $\SL({\mathbb E}(j)\mathop{\otimes}\limits_{{\mathbb K}_j} {\mathbb K})$, $j=2,...,t$. Lemma \ref{lem:GIT} implies that $w$ is $\varrho$-semistable.
\end{proof}
\subsection*{Asymptotic semistability on smooth curves}
Now, we assume that $C$ is smooth, i.e., that $C_i$ is smooth, $i=1,...,t$, and $C_i\cap C_j=\varnothing$, $1\le i<j\le t$. A torsion free sheaf on $C$ is a tuple $\E=(E(1),...,E(t))$ in which $E(i)$ is a locally free $\Oh_{C_i}$-module, $i=1,...,t$. Slightly abusing notation, we will view $E(i)$ as a torsion free sheaf on $C$, so that 
\[
\E=E(1)\oplus\cdots\oplus E(t).
\]
The \it slope \rm of $\E$ is 
\[
\ol\chi(\E):=\frac{\chi(\E)}{{\rm trk}(\E)}.
\]
So, we call $\E$ \it semistable\rm, if 
\[
\ol\chi(\F)\le \ol\chi(\E)
\]
holds for every saturated subsheaf $0\subsetneq \F\subsetneq \E$. 
\begin{Rem}
Let $i\in\{\, 1,...,t\,\}$. For a locally free $\Oh_{C_i}$-module $E$, one usually defines the slope as 
\[
\mu(E)=\frac{\deg(E)}{{\rm rk}(E)}.
\]
If we view $E$ as a torsion free sheaf on $C$, we have 
\[
\ol\chi(E)=\frac{1}{\ell_i}\cdot\bigl(\mu(E)+1-g(C_i)\bigr).
\]
\end{Rem}
For $i=1,...,t$, $E(i)$ is both a sub and a quotient object of $\E$. The following is an immediate consequence.
\begin{Lem}
\label{lem:HNFPrep}
A torsion free sheaf $\E=(E(1),...,E(t))$ on $C$ is semistable if and only if, for every index $i\in \{\, 1,...,t\,\}$, $E(i)$ is either zero or a non-zero semistable locally free $\Oh_{C_i}$-module with
\[
\ol\chi\bigl(E(i)\bigr)=\ol\chi(E).
\]
\end{Lem}
A torsion free $\Oh_C$-module admits a Harder--Narasimhan filtration
\[
0=:\E_0\subsetneq\E_1\subsetneq \cdots\subsetneq \E_s\subsetneq \E_{s+1}:=\E.
\]
We set 
\[
\ol\E_i:=\E_i/\E_{i-1},\q i=1,...,s+1.
\]
Then, $\ol\E_i$ is semistable, $i=1,...,s+1$, and 
\[
\ol\chi(\ol\E_1)>\cdots>\ol\chi(\ol\E_{s+1}).
\]
Moreover, we let 
\[
\ol\chi_{\max}(\E):=\ol\chi(\ol\E_1)\q\hbox{and}\q \ol\chi_{\min}(\E):=\ol\chi(\ol\E_{s+1}).
\]
\begin{Rem}
For $i\in\{\, 1,...,t\,\}$ and a locally free $\Oh_{C_i}$-module $E$, one has likewise a Harder--Narasimhan filtration
\[
0=:E_0\subsetneq E_1\subsetneq \cdots\subsetneq E_u\subsetneq E_{u+1}:=E
\]
in which $\ol E_i:=E_i/E_{i-1}$ is semistable, $i=1,...,u+1$, and 
\[
\mu(\ol E_1)>\cdots>\mu(\ol E_{u+1}).
\]
In this situation,
\[
\mu_{\max}(E):=\mu(\ol E_1)\q\hbox{and}\q \mu_{\min}(E):=\mu(\ol E_{u+1}).
\]
\end{Rem}
Now, Lemma \ref{lem:HNFPrep} generalizes to the following statement.
\begin{Lem}
\label{lem:HNFPrep2}
For a torsion free sheaf $\E=(E(1),...,E(t))$ on $C$, the identities
\[
\ol\chi_{\max}(\E)=\max\Bigl\{\,\frac{1}{\ell_i}\cdot \bigl(\mu_{\max}\bigl(E(i)\bigr)+1-g(C_i)\bigr) \,\Big|\, i=1,...,t\,\Bigr\}
\]
and
\[
\ol\chi_{\min}(\E)=\min\Bigl\{\,\frac{1}{\ell_i}\cdot \bigl(\mu_{\min}\bigl(E(i)\bigr)+1-g(C_i)\bigr) \,\Big|\, i=1,...,t\,\Bigr\}
\]
hold true.
\end{Lem}
Next, let $(\E,\ol q)$ be a torsion free $\Oh_C$-module endowed with a generalized parabolic structure and fix $\ul \kappa=(\kappa_1,...,\kappa_c)$ as before. The \it $\ul\kappa$-slope \rm of $(\E,\ul q)$ is 
\[
\ol\chi_{\ul\kappa}(\E,\ul q):=\frac{\chi_{\ul\kappa}(\E)}{{\rm trk}(\E)}.
\]
So, we call $(\E,\ul q)$ \it $\ul\kappa$-semistable\rm, if 
\[
\ol\chi_{\ul\kappa}(\F,\ul q_\F)\le \ol\chi_{\ul\kappa}(\E,\ul q)
\]
holds for every saturated subsheaf $0\subsetneq \F\subsetneq \E$. For every pair $(\E,\ul q)$, there is also a Harder--Narasimhan filtration 
\[
0=:(\E_0,\ul q_0)\subsetneq(\E_1,\ul q_1)\subsetneq \cdots\subsetneq (\E_s,\ul q_s)\subsetneq (\E_{s+1},\ul q_{s+1}):=(\E,\ul q)
\]
in which $(\ol \E_i,\ul r_i):=(\E_i,\ul q_i)/(\E_{i-1}, \ul q_{i-1})$ is $\ul\kappa$-semistable, $i=1,...,s+1$, and 
\[
\ol\chi_{\ul\kappa}(\ol\E_1,\ul r_1)>\cdots>\ol\chi_{\ul\kappa}(\ol\E_{s+1},\ul r_{s+1}).
\]
As before, we define
\[
\ol\chi_{\ul\kappa,\max}(\E,\ul q):=\ol\chi_{\ul\kappa}(\ol\E_1,\ul r_1)\q\hbox{and}\q \ol\chi_{\ul\kappa,\min}(\E,\ul q):=\ol\chi_{\ul\kappa}(\ol\E_{s+1}, \ul r_{s+1}).
\]
\begin{Lem}
\label{lem:CompareMaxSlopes}
Let $\E$ be a torsion free $\Oh_C$-module of uniform rank $r$, $\ul q$ a generalized parabolic structure on $\E$, $\ul t(\ul q)=(t_1,...,t_c)$ the type of $\ul q$, and $\ul\kappa$ a tuple of positive real numbers. Define 
\[
D:=r\cdot \Bigl(\sum_{i=1}^t \ell_i\Bigr)\cdot \Bigl(\sum_{j=1}^c \kappa_j\cdot t_j\Bigr).
\]
Then,
\[
\ol\chi_{\max}(\E)-D\le \ol\chi_{\ul\kappa,\max}(\E,\ul q)\le \ol\chi_{\max}(\E)
\]
and 
\[
\ol\chi_{\min}(\E)\le \ol\chi_{\ul\kappa,\min}(\E,\ul q)\le \ol\chi_{\min}(\E)+D.
\]
\end{Lem}
\begin{proof}
We verify the statement for the maximal slope. Let $\F\subset \E$ be a saturated subsheaf. It inherits a generalized parabolic structure $\ul q_\F$. Now,
\[
\ol\chi(\F)-D\le \ol\chi_{\ul\kappa}(\F,\ul q_\F)\le \ol\chi_{\ul\kappa,\max}(\E,\ul q).
\]
This proves the first inequality. On the other hand, 
\begin{eqnarray*}
\ol\chi_{\ul\kappa}(\F,\ul q_\F)&=&\ol\chi(\F)-\frac{1}{{\rm trk}(\F)}\cdot \sum_{i=1}^c\dim_\C(S_i)
\\
&\le& \ol\chi_{\max}(\E)-\frac{1}{{\rm trk}(\F)}\cdot \sum_{i=1}^c\dim_\C(S_i)
\\
&\le&  \ol\chi_{\max}(\E).
\end{eqnarray*}
This gives the second inequality.
\end{proof}
The main auxiliary result of this note is the following result.
\begin{Thm}
\label{thm:MainAux}
Fix $r>0$, $\chi\in\Z$, and a tuple $\ul\kappa=(\kappa_1,...,\kappa_c)$. Then, there exists a constant $K$, such that, for every $\delta>0$, every $(\ul\kappa,\delta)$-semistable $(a,b)$-GPS-swamp $(\E=(E(1),...,E(t)),\ul q,\varphi)$ in which $\E$ has uniform rank $r$ and $\chi(\E)=\chi$, every index $i\in\{\, 1,...,t\,\}$, and every subbundle $0\subsetneq F\subsetneq E(i)$, the inequality
\[
\mu(F)\le K
\]
is satisfied.
\end{Thm}
\begin{proof}
This corresponds to Theorem 3.3.20 in \cite{MunThesis}. We present a simpler proof which is a modification of the proof of Theorem 4.2.1 in \cite{GLSS}. Let 
\[
0=:(\E_0,\ul q_0)\subsetneq(\E_1,\ul q_1)\subsetneq \cdots\subsetneq (\E_s,\ul q_s)\subsetneq (\E_{s+1},\ul q_{s+1}):=(\E,\ul q)
\]
be the Harder--Narasimhan filtration of $(\E,\ul q)$. We take the filtration $\E_\bullet:\ \E_0\subsetneq \E_1\subsetneq\cdots\subsetneq \E_s\subsetneq \E_{s+1}$. Then, for any choice of $m_\bullet=(m_1,...,m_s)\in (\Q_{>0})^{\times s}$, we have 
\[
\chi_{\ul\kappa}(\E_\bullet, m_\bullet; \varphi)<0.
\]
The assumption on the semistability of $(\E,\ul q, \varphi)$ implies 
\[
\forall m_\bullet\in (\Q_{>0})^{\times s}:\q \mu(\E_\bullet, m_\bullet; \ul q)>0.
\]
In the notation of Remark \ref{rem:FinMan}, we have $\ul r=({\rm trk}(\E_1),...,{\rm trk}(\E_{s+1}))\in R$ and $\E_\bullet$, $\varphi$ define an element $P\in{\mathcal P}(s)$. Let $\|\cdot\|$ be the euclidean norm on $\R^{\alpha}$ and $\|\cdot\|_\infty$ the maximum norm. The function $\Phi(\ul r,P)$ takes on a minimal value on $\{\, \ul v\in {\mathcal C}(\ul r)\,|\, \|\ul v\|=1\,\}$. According to \cite{Ke}, Lemma 2.3, the minimal value occurs on a ray that is spanned by an integral vector. So, the minimal value must be positive. Since there are only finitely many possibilities for $\Phi(\ul r,P)$, we may find a positive constant $K_0>0$, such that the minimal value is at least $K_0$.
\par 
Suppose $(i_1,...,i_a)\in \{\,1,...,s+1 \,\}^{\times a}$ is an index for which the restriction of $\varphi$ to $(\E_{i_1}\otimes\cdots\otimes \E_{i_a})^{\oplus b}$ is non-zero. Then, we find an index $j_0\in \{\, 1,...,t\,\}$ and a non-zero homomorphism
\[
E_{i_1}(j_0)\otimes\cdots\otimes E_{i_a}(j_0)\lra \Oh_{C_{j_0}}.
\]
This implies
\[
\mu_{\min}\bigl(E_{i_1}(j_0)\bigr)+\cdots+\mu_{\min}\bigl(E_{i_a}(j_0)\bigr)=\mu_{\min}\bigl(E_{i_1}(j_0)\otimes\cdots\otimes E_{i_a}(j_0)\bigr)\le 0.
\]
With $\ol\chi_{\min}(E_k(j_0)):=(1/\ell_{j_0})\cdot (\mu_{\min}(E_k(j_0))+1-g(C_{j_0}))$, $k=i_1,...,i_a$, we infer 
\[
\ol\chi_{\min}\bigl(E_{i_1}(j_0)\bigr)+\cdots+\ol\chi_{\min}\bigl(E_{i_a}(j_0)\bigr) \le a.
\]
Lemma \ref{lem:HNFPrep2} implies 
\[
\ol\chi_{\min}(\E_{i_1})+\cdots+\ol\chi_{\min}(\E_{i_a})\le a.
\]
Using Lemma \ref{lem:CompareMaxSlopes}, we find
\begin{eqnarray}
\nonumber
\ol\chi_{\ul\kappa}(\ol\E_{i_1},\ul r_{i_1})+\cdots+\ol\chi_{\ul\kappa}(\ol\E_{i_a},\ul r_{i_a})
&=& \ol\chi_{\ul\kappa,\min}(\E_{i_1},\ul q_{i_1})+\cdots+\ol\chi_{\ul\kappa,\min}(\E_{i_a},\ul q_{i_a})
\\
\label{eq:UpperBoundMu}
&\le& a\cdot (D+1).
\end{eqnarray}
Set 
\[
K_1:=a \cdot (\chi+D+1).
\]
The vector $\Gamma_\bullet\in {\mathcal C}(\ul r)$ is given as 
\[
\bigl(\underbrace{\ol\chi_{\kappa}(\ol\E_1,\ul r_1)-\chi,...,\ol\chi_{\kappa}(\ol\E_1,\ul r_1)-\chi}_{{\rm trk}(\ol\E_1)\times},...,\underbrace{\ol\chi_{\kappa}(\ol\E_{s+1},\ul r_{s+1})-\chi,...,\ol\chi_{\kappa}(\ol\E_{s+1},\ul r_{s+1})-\chi}_{{\rm trk}(\ol\E_{s+1})\times}\bigr).
\]
Introducing $m_\bullet=(m_1,...,m_s)$ by
\[
m_i:=\frac{\ol\chi_{\kappa}(\ol\E_{i+1},\ul r_{i+1})-\ol\chi_{\kappa}(\ol\E_i,\ul r_i)}{\alpha},\q i=1,...,s,
\]
we have 
\[
\Gamma_\bullet=\sum_{i=1}^s m_i\cdot \Gamma^{{\rm trk}(\E_i)}.
\]
By \eqref{eq:UpperBoundMu},
\[
\mu(\E_\bullet,m_\bullet;\varphi)\le K_1.
\]
On the other hand,
\[
\mu(\E_\bullet,m_\bullet;\varphi)=\Phi(\ul r,P)(\Gamma_\bullet)=
\|\Gamma_\bullet\|\cdot \Phi(\ul r,P)\biggl(\frac{\Gamma_\bullet}{\|\Gamma_\bullet\|}\biggr)\ge \|\Gamma_\bullet\|\cdot K_0.
\]
So,
\[
\|\Gamma_\bullet\|_\infty \le \sqrt{\alpha}\cdot \|\Gamma_\bullet\|\le \sqrt{\alpha}\cdot \frac{K_1}{K_0}.
\]
This implies a bound on $\ol\chi_{\ul\kappa,\max}(\E,\ul q)$. In view of Lemma \ref{lem:HNFPrep2} and \ref{lem:CompareMaxSlopes}, this proves our claim.
\end{proof}
\begin{Cor}
\label{cor:MainAux}
Fix $r>0$, $\chi\in\Z$, and a tuple $\ul\kappa=(\kappa_1,...,\kappa_c)$. Then, there exists a positive rational number $\delta_\infty$, such that, for $\delta\ge\delta_\infty$, an $(a,b)$-GPS-swamp $(\E=(E(1),...,E(t)),\ul q,\varphi)$ in which $\E$ has uniform rank $r$ and $\chi(\E)=\chi$ is $(\ul\kappa,\delta)$-(semi)stable if and only if it is asymptotically $\ul\kappa$-(semi)stable.
\end{Cor}
\begin{proof}
The proof of Corollary 4.2.2 in \cite{GLSS} may be easily adapted to the situation of this corollary in order to infer it from the preceding theorem.
\end{proof}
\subsection*{Asymptotic semistability on nodal curves}
In this section, $C$ will be a connected projective curve with at most nodes as singularities. The set $N$ is assumed to be empty. For this reason, $\ul\kappa$ does not occur and will not appear in the notation. The same applies to the acronym ``GPS''.
\begin{Thm}
\label{thm:MainAux2}thm:MainAux
Fix $r>0$, $\chi\in\Z$.
\par 
{\rm i)} There exists a constant $H$, such that, for every $\delta>0$, every $\delta$-semistable $(a,b)$-swamp $(\E,\varphi)$ in which $\E$ has uniform rank $r$ and $\chi(\E)=\chi$,  and every saturated subsheaf $0\subsetneq \F\subsetneq \E$, the inequality
\[
\ol\chi(\F)\le H
\]
is satisfied.
\par 
{\rm ii)} There exists a positive rational number $\delta_\infty$, such that, for $\delta\ge\delta_\infty$, an $(a,b)$-swamp $(\E,\varphi)$ in which $\E$ has uniform rank $r$ and $\chi(\E)=\chi$ is $\delta$-(semi)stable if and only if it is asymptotically (semi)stable.
\end{Thm}
\begin{proof}
The proof of these assertions rests on a construction due to Bhosle \cite{Bhosle} which has been exploited in the realm of principal bundles on irreducible nodal curves in \cite{SchNod}, \cite{SchNod2}, and on reducible nodal curves in \cite{MunThesis}, \cite{MunPap2}. Let us recall some of the details. Look at the normalization $\nu\colon \widetilde{C}\lra C$ of $C$ and suppose $\E$ is a torsion free sheaf on $C$. Then,
$\nu^\star(\E)$ may have torsion. However, the torsion subsheaf of $\E$ may be described in terms of the stalks of $\E$ at the nodes, using Seshadri's classification of the stalks of torsion free sheaves at nodes (\cite{Sesh}, huiti\`eme partie, Proposition 2).
Next, we form $\F:=\nu^\star(\E)/\ul{\rm Tors}(\nu^\star(\E))$. The description of the torsion submodule of $\nu^\star(\E)$ shows that there are only finitely many possibilities for the Euler characteristic of $\F$ (\cite{SchNod2}, Proposition 3.1, \cite{MunThesis}, Theorem 4.2.10, \cite{MunPap2}, Proposition 5.3). On $C$, we now get the short exact sequence 
\begin{equation}
\label{eq:PresentTF}
\begin{CD}
0 @>>> \E @>>> \nu_\star(\F) @>>> {\mathcal R} @>>> 0.
\end{CD}
\end{equation}
Let $N_1,...,N_c$ be the nodes of $C$, and $\mu_i,\nu_i$ the two points of $\widetilde{C}$ that lie over $N_i$, $i=1,...,c$. Then, $R_i:={\mathcal R}\otimes \C(N_i)$ is a finite dimensional $\C$-vector space, and $\nu_\star(\F)\otimes \C(N_i)$ may be identified with $\F_{|\{\mu_i\}}\oplus \F_{|\{\nu_i\}}$. In other words, the above short exact sequence provides us with a generalized parabolic structure $\ul q$ on $\F$. Again, the dimension of $R_i$ is determined by the structure of the stalk of $\E$ at $N_i$, $i=1,...,t$ (\cite{SchNod2}, Proposition 3.1, \cite{MunThesis}, Theorem 4.2.10, \cite{MunPap2}, Proposition 5.3). A homomorphism $\varphi\colon (\E^{\otimes a})^{\oplus b}\lra \Oh_C$ induces a homomorphism $\widetilde{\varphi}\colon (\F^{\otimes a})^{\oplus b}\lra \Oh_{\widetilde{C}}$.
Note that any weighted filtration $(\E_\bullet,m_\bullet)$ of $\E$ induces a weighted filtration $(\F_\bullet, m_\bullet)$ of $\F$. Since the total ranks of $\E$ and $\F$ are equal, it follows readily that 
\[
\mu(\F_\bullet, m_\bullet;\widetilde{\varphi})=\mu(\E_\bullet, m_\bullet;\varphi).
\]
Furthermore, setting $\ul 1=(1,...,1)$, it is also true that 
\[
\chi_{\ul 1}(\F_\bullet,m_\bullet)=\chi(\E_\bullet,m_\bullet).
\]
This is immediate from the definitions and was originally used by Bhosle \cite{Bhosle}, Proposition 4.2, in the context of semistable torsion free sheaves on nodal curves, by the second author in \cite{SchNod}, Proposition 5.2.2, \cite{SchNod2}, Proposition 3.2, in the context of tensor fields and principal bundles on irreducible nodal curves, and by the first author in \cite{MunThesis}, Lemma  4.4.4, \cite{MunPap2}, Proposition 5.9, in the context of tensor fields and principal bundles on reducible nodal curves. These identities show that, for $\delta>0$, the $(a,b)$-GPS-swamp $(\F,\ul q,\widetilde{\varphi})$ is $(\ul 1,\delta)$-semistable if and only if the swamp $(\E_\bullet,\varphi)$ is $\delta$-semistable. We fixed the total rank $r$ and the Euler characteristic $\chi$ of $\E$. This implies that $\F$ has also total rank $r$, and, as pointed out before, there are only finitely many possibilities for the Euler characteristic of $\F$. In view of these facts, the theorem follows immediately from Theorem \ref{thm:MainAux} and Corollary \ref{cor:MainAux}.
\end{proof}

\section{Singular principal bundles on nodal curves}\label{sect:SingPrinz}

As in the last section, we let $C$ be a connected projective curve with at most nodes as singularities. Suppose that $G$ is a semisimple linear algebraic group and $\sigma\colon G\lra \SL(V)$ a faithful representation. A \it pseudo $G$-bundle \rm is a pair $(\E,\tau)$ which consists of a torsion free $\Oh_C$-module $\E$ of uniform rank $r$, $r:=\dim_\C(V)$, and a homomorphism
\[
\tau\colon \ul{\rm Sym}^\star(\E\otimes V)^G\lra \Oh_C
\]
of $\Oh_C$-algebras. We set 
\[
\ul{\rm Hom}(V\otimes\Oh_C,\E^\vee)\catqot G:=\ul{\rm Spec}\bigl(\ul{\rm Sym}^\star(\E\otimes V)^G\bigr).
\]
Then, the datum of $\tau$ corresponds to the datum of a section 
\[
s\colon C\lra \ul{\rm Hom}(V\otimes\Oh_C,\E^\vee)\catqot G.
\]
Let $C^\circ\subset C$ be the complement of the nodes. The sheaf $\E_{|C^\circ}$ is locally free of rank $r$ and 
\[
\ul{\rm Isom}(V\otimes\Oh_{C^\circ},\E_{|C^\circ}^\vee)/G\subset \ul{\rm Hom}(V\otimes\Oh_C,\E^\vee)\catqot G
\]
is the frame bundle of the locally free sheaf $\E^\vee_{|C^\circ}$ on $C^\circ$. A pseudo $G$-bundle $(\E,\tau)$ is said to be a \it singular principal $G$-bundle\rm, if 
\[
U(\E,\tau):=s^{-1}(\ul{\rm Isom}(V\otimes \Oh_{C^\circ},\E_{|C^\circ})/G)\subset C^\circ 
\]
is dense in $C$. By means of the base change diagram
$$
\xymatrix{
\mathcal{P}(\E,\tau) \ar[rr]\ar[d] && \ul{\rm Isom}(V\otimes\Oh_{C^\circ},\E_{|C^\circ}^\vee)\ar[d]
\\
U(\E,\tau) \ar[rr]^-{s_{|U(\E,\tau)}} && \ul{\rm Isom}(V\otimes\Oh_{C^\circ},\E_{|C^\circ}^\vee)/G
}
$$
we obtain the principal $G$-bundle ${\mathcal P}(\E,\tau)$ on $U(\E,\tau)$.
\begin{Rem}
As before, we let ${\mathbb K}_i$ be the function field of $C_i$ and $\eta_i:={\rm Spec}({\mathbb K}_i)$ the generic point of that component, $i=1,...,t$. A pseudo $G$-bundle $(\E,\tau)$ is a singular principal $G$-bundle if and only if 
\[
s(\eta_i)\in {\rm Isom}\bigl(V\mathop{\otimes}_\C{\mathbb K}_i,{\mathbb E}(i)\bigr)/G_{{\mathbb K}_i},\q i=1,...,t.
\]
Here, ${\mathbb E}(i):=\E_{|\{\eta_i\}}$ and $G_{{\mathbb K}_i}:=G\mathop{\times}\limits_{{\rm Spec}(\C)} {\rm Spec}({\mathbb K}_i)$, $i=1,...,t$.
\end{Rem}
A one parameter subgroup $\la\colon \C^\star\lra G$ defines a parabolic subgroup $Q_G(\la)$ of $G$. Via $G\stackrel{\sigma}{\subset} \SL(V)\subset \GL(V)$, we may view $\la$ as a one parameter subgroup of $\GL(V)$ and get a parabolic subgroup $Q_{\GL(V)}(\la)$ ($\supset Q_G(\la)$) of $\GL(V)$. If $\la$ is the constant one parameter subgroup, then $Q_G(\la)=G$ and $Q_{\GL(V)}(\la)=\GL(V)$. Suppose that $\ul\la=(\la_1,...,\la_t)$ is a tuple in which $\la_i$ is a one parameter subgroup of $G$, $i=1,...,t$, and which is non-trivial in the sense that at least one of the entries of $\ul\la$ is non-constant and that $(\E,\tau)$ is a singular principal $G$-bundle. A \it reduction \rm of $(\E,\tau)$ to $\ul \la$ is given by a tuple $\ul\beta=(\beta_1,...,\beta_t)$ of sections 
\[
\beta_i\colon U_i'\lra {\mathcal P}(\E,\tau),
\]
$U_i'\subset U\cap C_i$ a non-empty open subset, $i=1,...,t$. 
\begin{Rem}
Again, it suffices to specify $\beta_i$ at the generic point $\eta_i$, $i=1,...,t$.
\end{Rem}
Now, we use the notation introduced in the proof of Lemma \ref{lem:GenSemStab}. So, we let ${\mathbb E}=({\mathbb E}(1),...,{\mathbb E}(t))$ with ${\mathbb E}(i)=\E_{|\{\eta_i\}}$, $i=1,...,t$.
The reduction $\beta_{i|\{\eta_i\}}$ and $\la_i$ define a weighted filtration $({\mathbb E}^i_\bullet,\Gamma^i_\bullet)$ of ${\mathbb E}(i)$, $i=1,...,t$. (If $\la_i$ is constant, then $({\mathbb E}^i_\bullet,\Gamma^i_\bullet)=(0\subsetneq {\mathbb E}(i), (0))$.) Since $\sigma\circ \la_i$ is a one parameter subgroup of the 
special linear group, the process described in the proof of Lemma \ref{lem:GenSemStab} gives a weighted filtration $(\E({\ul\beta})_\bullet, m({\ul\beta})_\bullet)$ of $\E$.
\par 
A singular principal $G$-bundle $(\E,\tau)$ is \it (semi)stable\rm, if for every non-trivial tuple $\ul\la=(\la_1,...,\la_t)$ of one parameter subgroups of $G$ and every reduction $\ul\beta=(\beta_1,...,\beta_t)$ of $(\E,\tau)$ to $\ul\la$, the inequality
\[
\chi(\E({\ul\beta})_\bullet, m({\ul\beta})_\bullet) (\ge) 0
\]
is satisfied.
\begin{Thm}\label{thm:moduli-stb}
Let $G$ be a semisimple linear algebraic group, $\sigma\colon G\lra \SL(V)$ a faithful representation of $G$, $C$ a connected nodal projective curve, $\L$ an ample line bundle on $C$, and $\chi\in\Z$ an integer. Then, there exists a projective moduli scheme ${\mathcal S\mathcal P\mathcal B}_{\sigma/C/\L/\chi}$ for semistable principal $G$-bundles $(\E,\tau)$ with $\chi(\E)=\chi$. 
\end{Thm}
\begin{proof}
The first author defined a notion of $\delta$-semistability for pseudo $G$-bundles on a reducible nodal curve, $\delta\in \Q_{>0}$. The moduli space mentioned in the theorem is the moduli space ${\mathcal S\mathcal P\mathcal B}(\sigma)_\chi^{\delta\hbox{-}\rm ss}$ of $\delta$-semistable pseudo $G$-bundles from \cite{MunPap} for large values of $\delta$. In order to understand this, let us discuss some elements of the construction.
\par 
There exist positive integers $a$ and $b$, such that one may assign to every pseudo $G$-bundle $(\E,\tau)$ an $(a,b)$-swamp $(\E,\varphi_\tau)$. For this note, that, for a torsion free $\Oh_C$-module $\E$ and $d>0$, the Reynolds operator yields a projection
\[
(\E\otimes V)^{\otimes d} \lra \ul{\rm Sym}^d(\E\otimes V)\lra \ul{\rm Sym}^d(\E\otimes V)^G.
\]
Since invariant rings are finitely generated, $\tau$ is determined by a homomorphism
\[
\psi_\tau\colon \bigoplus_{d=1}^e (\E\otimes V)^{\otimes d}\lra \Oh_C,
\]
$e>0$ sufficiently large. Homogenizing $\psi_\tau$ yields $\varphi_\tau$. We refer to \cite{MunPap}, \S 5, for the details.
\par 
Now, the assignment $(\E,\tau)\lma (\E,\varphi_\tau)$ is injective on the level of isomorphy classes and $(\E,\tau)$ is $\delta$-(semi)stable if and only if the associated $(a,b)$-swamp $(\E,\varphi_\tau)$ is $\delta$-(semi)stable in the sense of Section \ref{sec:ParSwa}, $\delta\in \Q_{>0}$.
So, by Theorem \ref{thm:MainAux2}, we need to show that a) $(\E,\tau)$ is a singular principal $G$-bundle if and only if $\varphi_\tau$ is generically semistable and that b), for a weighted filtration $(\E_\bullet, m_\bullet)$ of $\E$, the condition $\mu(\E_\bullet,m_\bullet;\varphi_\tau)=0$ is equivalent to the condition that there are a non-trivial tuple $\ul\la=(\la_1,...,\la_t)$ of one parameter subgroups of $G$ and a reduction $\ul\beta$ of $(\E,\tau)$ to $\ul\la$ with $(\E_\bullet, m_\bullet)=(\E(\ul\beta)_\bullet, m(\ul\beta)_\bullet)$.
\par 
As far as a) is concerned, we have seen in the proof of Lemma \ref{lem:GenSemStab} that $\varphi_\tau$ is generically semistable if and only if, for every index $i\in \{\, 1,...,t\,\}$, the point $w_i'$ is non-zero and semistable with respect to the action of $\SL({\mathbb E}(i))$ on $W_i'$. By \cite{GLSS}, Lemma 2.1.3, this is equivalent to the property that there exists a non-empty open subset $U_i'\subset C^\circ\cap C_i$ with $s(U_i')\subset \ul{\rm Isom}(V\otimes\Oh_{C^\circ},\E_{|C^\circ}^\vee)/G$, $i=1,...,t$.
\par 
For b), we observe that the proofs of Lemma \ref{lem:GIT} and Lemma \ref{lem:GenSemStab} show that 
\[
\mu(\E_\bullet, m_\bullet;\varphi)=\max\bigl\{\,\mu(\la_i, w_i')\,|\, i=1,...,t\,\bigr\}=0
\]
happens if and only if $\la_i$ is a one parameter subgroup of $\SL({\mathbb E}(i))$ and $\mu(\la_i,w_i')=0$, $i=1,...,t$. This includes the possibility that $\la_i$ is constant, for some of the indices $i\in\{\, 1,...,t\,\}$. Proposition 2.1.4 in \cite{GLSS} shows that one may assume without loss of generality that $\la_i$ is conjugate to a one parameter subgroup $\la_i'$ of $G$, so that $({\mathbb E}^i_\bullet, \Gamma^i_\bullet)$ comes from a reduction $\beta_i\colon U'_i\lra {\mathcal P}(\E,\tau)/Q_G(\la_i')$, for some open subset $U_i'\subset U\cap C_i$, $i=1,...,t$.
\end{proof}
\subsection*{Good choices of the faithful representation}
Suppose that $b\colon V\mathop{\otimes}\limits_\C V\lra \C$ is a non-degenerate bilinear form. Let $\psi\colon V \lra V^\vee$, $v\lma (w\lma b(v,w))$, be the corresponding automorphism. We say that $\sigma$ \it fixes \rm $b$, if 
\[
\forall g\in G\forall v,w\in V:\q b\bigl(\sigma(g)(v),\sigma(g)(w)\bigr)=b(v,w).
\]
We will work with the commutative diagram
$$
\xymatrix{
{\ul{\rm Hom}(V\otimes \Oh_X,\E^\vee)}\ar[rr]^-{q}\ar[dr]^-{p} &&
{\ul{\rm Hom}(V\otimes \Oh_X,\E^\vee)\catqot G} \ar[dl]_-{\ol p}
\\
& X &
}
$$
By the universal property of 
\[
{\mathcal H}:=\ul{\rm Hom}(V\otimes \Oh_X,\E^\vee):=\ul{\rm Spec}\bigl(\ul{\rm Sym}^\star(\E\otimes V)\bigr),
\]
there is a tautological homomorphism
\[
p^\star\bigl(\ul{\rm Sym}^\star(\E\otimes V)\bigr)\lra \Oh_{{\mathcal H}}
\]
of $\Oh_{{\mathcal H}}$-algebras. The restriction of this homomorphism to the component of degree one yields the homomorphism
\[
p^\star(\E)\lra V^\vee\otimes \Oh_{{\mathcal H}}.
\]
Let $\widetilde{b}\colon V^\vee\mathop{\otimes}\limits_\C V^\vee \lra \C$ be the bilinear form defined by $\psi^{-1}$. Then, the above homomorphism and $\widetilde{b}$ yield a bilinear pairing 
\[
\widetilde{b}_{\rm taut}\colon p^\star(\E\otimes \E)\cong p^\star(\E)\otimes p^\star(\E)\lra \Oh_{\mathcal H}.
\]
Push this forward to $\ul{\rm Hom}(V\otimes \Oh_C,\E^\vee)\catqot G$ and note that 
\[
p^\star(\E\otimes \E)\cong q^\star\bigl(\ol p^\star(\E\otimes \E)\bigr).
\]
So, we get 
\[
\ol p^\star(\E\otimes \E)\otimes q_\star(\Oh_{\mathcal H}) \lra q_\star(\Oh_{\mathcal H}).
\]
Now,
\[
\Oh_{{\mathcal H}\catqot G}=q_\star(\Oh_{\mathcal H})^G.
\]
By the existence and properties of the Reynolds operator, $\Oh_{{\mathcal H}\catqot G}$ is a direct summand of the $\Oh_{{\mathcal H}\catqot G}$-module $q_\star(\Oh_{\mathcal H})$. This means that we obtain a homomorphism
\[
\ol{b}_{\rm taut}\colon \ol p^\star(\E\otimes \E)\lra \Oh_{{\mathcal H}\catqot G}
\]
of $\Oh_{{\mathcal H}\catqot G}$-modules. The assumption on $\sigma$ and $b$ implies that the pullback of $\ol{b}_{\rm taut}$ via $q$ agrees over $p^{-1}(C^\circ)$ with $\widetilde{b}_{\rm taut}$.
Pulling back $\ol{b}_{\rm taut}$ via the section $s$ yields the bilinear form 
\[
\widetilde{b}(\E,\tau)\colon \E\otimes \E\lra \Oh_C.
\]
Our last observation implies that the restriction of $\widetilde{b}(\E,\tau)$ to $U(\E,\tau)$ is non-degene\-rate. So, the induced homomorphism
\[
\psi(\E,\tau)\colon \E\lra \E^\vee
\]
is an isomorphism over $U(\E,\tau)$. Since $\E$ is torsion free, the homomorphism $\psi(\E,\tau)$ is injective
\begin{Prop}
\label{prop:EulerGleich}
Assume that $\chi(C,\E)=\dim_\C(V)\cdot \chi(C,\Oh_C)$, Then, the homomorphism $\psi(\E,\tau)$ is an isomorphism.
\end{Prop}
\begin{proof}
Proposition \ref{prop:PrepareIso}, ii), in the appendix shows that $\chi(C,\E)=\chi(C,\E^\vee)$. This implies that the injective homomorphism $\psi(\E,\tau)\colon \E\lra \E^\vee$ is also surjective.
\end{proof}
\begin{Thm}\label{thm:good-choice}
Suppose that there is a non-degenerate bilinear form $b\colon  V\mathop{\otimes}\limits_\C V\lra \C$ which is fixed by $\sigma$. Then, 
\[
U(\E,b)=C^\circ
\]
holds for every singular principal $G$-bundle $(\E,\tau)$ on $C$ with $\chi(C,\E)=\dim_\C(V)\cdot \chi(C,\Oh_C)$. 
\end{Thm}
\begin{proof}
Let $(\E,\tau)$ be a singular principal $G$-bundle on $C$ and $c\in C^\circ$. We will use ${}_-\langle c\rangle$ as a shorthand notation for ${}_-\mathop{\otimes}\limits_{\Oh_{C,c}} \C(c)$. By construction, $\psi(\E,\tau)\langle c\rangle$ factorizes as 
\[
\E\langle c\rangle \stackrel{h^\vee}{\lra} V^\vee\stackrel{\psi^{-1}}{\lra} V\stackrel{h}{\lra}\E^\vee\langle c\rangle.
\]
Here, $h\in {\rm Hom}(V,\E^\vee\langle c\rangle)$ is a homomorphism which maps to $s(c)$ under the quotient map $q$. Obviously, $h$ must be an isomorphism, i.e., $c\in U(\E,\tau)$.
\end{proof}

\section*{Summary and outlook}

The results of the present note and \cite{MunPap2} fully generalize the results of \cite{SchNod2} to reducible nodal curves. The explicit estimates given in Theorem \ref{thm:MainAux} and the observations on good choices of the faithful representation of the structure group together with the construction in \cite{MunUniv} yield a reasonable generalization of Pandharipande's moduli space \cite{Pan}. Possible applications include the study of conformal blocks on stable curves and modular interpretations of the degenerations obtained by Manon \cite{Man} and Belkale and Gibney \cite{BG}. It will also be interesting to investigate, if the techniques developed by Balaji \cite{Bal} may be extended to get a Hilbert compactification which generalizes the work \cite{SchHilb} by the second author for $\GL_r(\C)$.

\section*{Appendix}

\renewcommand\thesection{\Alph{section}}
\setcounter{section}{1}
\setcounter{Ex}{0}
We work in the setting outlined at the beginning of Section \ref{sec:ParSwa}, except that we won't make use of the set $N$. Let us fix an invertible $\Oh_C$-module ${\mathcal N}$. For a coherent $\Oh_C$-module $\G$, we define the \it ${\mathcal N}$-dual \rm 
of $\G$ as $\G_{\mathcal N}^\vee:=\ul{\rm Hom}(\G,{\mathcal N})$, and the \it ${\mathcal N}$-reflexive hull \rm of $\G$ as $\G^{\vee\vee}_{\mathcal N}:=(\G^\vee_{\mathcal N})^\vee_{\mathcal N}$. As usual, we write $\G^\vee$ instead of $\G^\vee_{\Oh_C}$. Since $C$ is a Gorenstein variety, $\omega_C$ is also an  invertible $\Oh_C$-module. In this case, we write $\G^\circ:=\G^\vee_{\omega_C}$ and $\G^{\circ\circ}:=\G^{\vee\vee}_{\omega_C}$. For any  invertible $\Oh_C$-module ${\mathcal N}$, there is a canonical homomorphism $\iota_{\mathcal N}\colon \G\lra \G^{\vee\vee}_{\mathcal N}$. Its kernel is the torsion submodule of $\G$.
\begin{Rem}
\label{rem:ComputeDegTensor}
i) Since ${\mathcal N}$ is invertible, the natural homomorphism
\[
\ul{\rm Hom}(\G,\Oh_X)\mathop{\otimes}_{\Oh_X} {\mathcal N} \lra \ul{\rm Hom}(\G,{\mathcal N})
\]
is an isomorphism.
\par 
ii) Let ${\mathcal N}$ be an invertible $\Oh_X$-module and define $n_i:=\deg({\mathcal N}_{|C_i})$, $i=1,...,t$. For a torsion free $\Oh_X$-module $\E$, we get the short exact sequence \eqref{eq:PresentTF}. The corresponding sequence for $\E\otimes {\mathcal N}$ is \eqref{eq:PresentTF}$\otimes{\mathcal N}$, by the projection formula. This shows 
\[
\chi(X,\E\otimes {\mathcal N}) =\chi\Bigl(X,\nu_\star \bigl(\F\otimes \nu^\star({\mathcal N})\bigr)\Bigr)-\dim_\C({\mathcal R}).
\]
Since the normalization $\nu\colon \widetilde{X}\lra X$ is a finite morphism, we have 
\[
\chi\Bigl(X,\nu_\star \bigl(\F\otimes \nu^\star({\mathcal N})\bigr)\Bigr)
=
\chi\bigl(\widetilde{X}, \F\otimes \nu^\star({\mathcal N})\bigr).
\]
The Riemann--Roch theorem on smooth curves implies 
\[
\chi\bigl(\widetilde{X}, \F\otimes \nu^\star({\mathcal N})\bigr)=r\cdot\sum_{i=1}^t n_i +\chi(\widetilde{X},\F).
\]
Continuing this way, we compute
\begin{align*}
\chi(X,\E\otimes {\mathcal N}) 
&= r\cdot\sum_{i=1}^t n_i +\chi(\widetilde{X},\F)-\dim_\C({\mathcal R})
\\
&= r\cdot\sum_{i=1}^t n_i +\chi\bigl(X,\nu_\star(\F)\bigr)-\dim_\C({\mathcal R})
\\
&= r\cdot\sum_{i=1}^t n_i +\chi(X,\E).
\end{align*}
\par 
iii) Using the notation from the previous part for ${\mathcal N}=\omega_X$, the formula 
\[
\chi(X,\omega_X)=\sum_{i=1}^t n_i+\chi(X,\Oh_X)
\]
and the equality $\chi(X,\omega_X)=-\chi(X,\Oh_X)$ which follows from Serre duality show
\[
\sum_{i=1}^t n_i=-2\cdot \chi(X,\Oh_X).
\]
\end{Rem}
\begin{Prop}
\label{prop:PrepareIso}
Let ${\mathcal N}$ be an invertible $\Oh_C$-module and $\E$ a torsion free $\Oh_C$-module of uniform rank $r$.
\par 
{\rm i)} The $\Oh_C$-module $\E^{\vee}_{\mathcal N}$ is torsion free as well, and the canonical homomorphism $\iota_{\mathcal N}\colon \E\lra \E^{\vee\vee}_{\mathcal N}$ is an isomorphism.
\par 
{\rm ii)} One has $\chi(C,\E^\circ)=-\chi(C,\E)$ and $\chi(C,\E^\vee)= -\chi(C,\E)+2\cdot r\cdot \chi(C,\Oh_C)$. In particular, if $\chi(C,\E)=r\cdot \chi(C,\Oh_C)$, then $\chi(C,\E^\vee)= \chi(C,\E)$.
\end{Prop}
\begin{proof}
i) This is a local problem and, so, follows from \cite{Vasc}, Corollary (2.3). 
\par 
ii) By Serre duality, 
\[
\chi(C,\E) = h^{0}(C,\E)-h^{0}(C,\E^\circ),\q
\chi(C,\E^\circ) = h^{0}(C,\E^\circ)-h^{0}(C,\E^{\circ\circ}).
\]
So, the first equality follows immediately from Part i). The formulas from Remark \ref{rem:ComputeDegTensor} show  $\chi(X,\E^\circ)=\chi(X,\E^\vee)-2\cdot r\cdot \chi(X,\Oh_X)$. Plugging this into the first equality gives the second one. 
\end{proof}
\begin{Rem}
i) By Remark \ref{rem:EnforceRiemannRoch}, the condition $\chi(C,\E)=r\cdot \chi(C,\Oh_C)$ is equivalent to the condition ${\deg}_{\ul\ell}(\E)=0$. 
\par 
ii) Similar computations appear, for irreducible curves, even on non-Gorenstein ones, in Section 3.1 of the (unpublished) work \cite{Cook}.
\end{Rem}

\end{document}